\documentclass[twocolumn]{autart}
\usepackage{amsfonts}
\usepackage{amsmath}
\usepackage{amssymb}
\usepackage{algorithm}
\usepackage{bm} 
\usepackage{booktabs}
\usepackage[pdftex]{graphicx}
\usepackage{color}
\usepackage[normalem]{ulem}

\usepackage{enumitem}
\usepackage{epstopdf}

\makeatletter
\renewcommand*{\@opargbegintheorem}[3]{\trivlist
      \item[\hskip \labelsep{\bfseries #1\ #2}] \textbf{(#3)}\ \itshape}
\makeatother

\newtheorem{theorem}{Theorem}[section]
\newtheorem{proposition}[theorem]{Proposition}

\newtheorem{definition}[theorem]{Definition} 
\newtheorem{lemma}[theorem]{Lemma}
\newtheorem{remark}[theorem]{Remark}

\newtheorem{assumption}[theorem]{Assumption}
\newtheorem{proof}{Proof}

\newcommand{\ms}{{{\rm\bf q}}}
\renewcommand{\Pr}{{\mathbb{P}}}

\newcommand{\decision}{{\mathbf{x}}}
\newcommand{\solvemip}{{\texttt{Solve\textsubscript{MIP}}}}

\renewcommand{\natural}{{\mathbb{N}}}

\newcommand{\until}[1]{\{1,\ldots,#1\}}

\newcommand{\EE}{\mathcal{E}} \newcommand{\FF}{\mathcal{F}}

\newcommand{\NN}{\mathcal{N}}

\newcommand{\VV}{\mathcal{V}} 

\newcommand{\bQ}{{\mathbb{Q}}}

\newcommand{\bR}{{\mathbb{R}}}

\newcommand{\bZ}{{\mathbb{Z}}}

\newcommand{\cSet}[1]{{\mathcal{#1}}}

\newcommand{\MC}[1]{{\color{black} #1}}
\newcommand{\GN}[1]{{\color{black} #1}}

\newcommand{\subscript}[2]{$#1 _ #2$}

\newcommand\oprocendsymbol{\hbox{$\square$}}
\newcommand\oprocend{\relax\ifmmode\else\unskip\hfill\fi\oprocendsymbol}

\begin{document}
\begin{frontmatter}

\title{
Randomized Constraints Consensus for\\ Distributed Robust Mixed-Integer Programming\thanksref{footnoteinfo}}

\author[SUTD]{Mohammadreza Chamanbaz}\ead{Chamanbaz@sutd.edu.sg},
\author[unibo]{Giuseppe Notarstefano}\ead{giuseppe.notarstefano@unibo.it},
\author[unisalento]{Francesco Sasso}\ead{francesco.sasso@unisalento.it},
\author[SUTD]{Roland Bouffanais}\ead{bouffanais@sutd.edu.sg}

\address[SUTD]{Singapore University of Technology and Design, 8 Somapah Road Singapore 487372.}
\address[unisalento]{Department of Engineering, Universit\`{a} del
    Salento, Italy.}
\address[unibo]{Department of Electrical, Electronic and
    Information Engineering G. Marconi, University of Bologna, Bologna, Italy.}

\thanks[footnoteinfo]{This work is supported by the European Research Council
  (ERC) under the European Union's Horizon 2020 research and innovation
  programme, grant agreement No 638992 - OPT4SMART, and   a grant from
  the Singapore National Research Foundation (NRF) under the ASPIRE project,
  grant No NCR-NCR001-040. 
  \MC{A preliminary conference version of this paper appeared as \cite{CHAMANBAZ20174973}. This paper addresses a more general optimization set-up, includes a more in-depth analysis and new numerical computations.
  Corresponding author Mohammadreza Chamanbaz, Tel. +6584019212.}}

\begin{abstract}

  In this paper, we consider a network of processors aiming at cooperatively
  solving mixed-integer convex programs subject to uncertainty. Each node only
  knows a common cost function and its local uncertain constraint set. We
  propose a randomized, distributed algorithm working under asynchronous,
  unreliable and directed communication. The algorithm is based on a local
  computation and communication paradigm. At each communication round, nodes
  perform two updates: (i) a verification in which they check---in a randomized
  fashion---the robust feasibility  of \MC{a} candidate
  optimal point, and (ii) an optimization step in which they exchange their
  candidate basis (the minimal set of constraints defining a solution) with
  neighbors and locally solve an optimization problem.
  As main result, we show that processors can stop the algorithm after a finite
  number of communication rounds (either because verification has been
  successful for a sufficient number of rounds or because a given threshold has
  been reached), so that candidate optimal solutions are consensual.  The common
  solution is proven to be---with high confidence---feasible and hence optimal
  for the entire set of uncertainty except a subset having \MC{an} arbitrary small
  probability measure. We show the effectiveness of the proposed distributed
  algorithm using two examples: a random, uncertain mixed-integer linear program
  and a distributed localization in wireless sensor networks. The distributed
  algorithm is implemented on a multi-core platform in which the nodes
  communicate asynchronously.

\end{abstract}


\end{frontmatter}

\section{Introduction}\label{sec:introduction}

Robust optimization plays an important role in several areas such as
estimation and control and has been widely investigated. Its rich literature
dates back to the $1950$s, see~\cite{ben2009robust} and references therein. Very
recently, there has been a renewed interest in this topic in a parallel and/or
distributed frameworks.
In the early reference \cite{lee2013distributed}, a synchronous distributed
random projection algorithm with almost sure convergence was proposed for the
case where each node knows a local cost function and an uncertain constraint.
In \cite{lee2016asynchronous}\MC{,} the same authors extended the algorithm to a
gossip-based protocol (based on pairwise node selection). %
To prove almost sure convergence, the algorithms in
\cite{lee2013distributed,lee2016asynchronous} require suitable assumptions on
the random set, network topology and algorithmic weights.  
Authors in \cite{towfic_adaptive_2014} proposed a distributed approach based on
barrier methods. The algorithm needs to start from a strictly feasible point.
In \cite{carlone2014distributed}, a parallel/distributed scheme is considered
for solving an uncertain optimization problem by means of the scenario approach
\cite{calafiore_uncertain_2004,calafiore_scenario_2006}. The scheme consists of
extracting a number of samples from the uncertain set and assigning them to
nodes in a network. Each node is assigned a portion of the extracted
samples. Then, a variant of the constraints consensus algorithm introduced in
\cite{notarstefano2011distributed} is used to solve the deterministic
optimization problem.
A similar parallel framework based on the scenario approach is considered in
\cite{You2018} to solve convex optimization problems with a common uncertain
constraint. In this setup, the sampled optimization problem is solved in a
distributed way by using a primal-dual subgradient (resp. random projection)
algorithm over an undirected (resp. directed) graph.
In \cite{burger2014polyhedral}, a cutting plane consensus algorithm is
introduced for solving convex optimization problems where constraints are
distributed through the network processors and all processors have a common cost
function. In the case where constraints are uncertain, a worst-case approach
based on a pessimizing oracle is used. 
A distributed scheme based on the scenario approach is introduced in
\cite{margellos2016distributed}. Each node initially extracts random samples
from its local uncertain constraint set and a distributed proximal minimization
algorithm is designed to solve the sampled optimization problem.
Xie and coauthors proposed in \cite{You2018} a distributed algorithm for solving
convex programs over time-varying and unbalanced digraphs. To this end,
an epigraph form of the problem is considered to convert the objective function
into a linear form and a two-step recursive algorithm is used which asymptotically
converges to a solution. One of the assumptions in this setup is that each
node needs to know its identifier.
\MC{A distributed algorithm---based on an iterative scheme that combines dual decomposition and proximal minimization---is presented in \cite{FALSONE2017149} which takes into account privacy of information.}

All aforementioned papers deal with usual (continuous) optimization problems.
A centralized method---based on the scenario approach---is presented in
\cite{calafiore2012mixed} for solving robust mixed-integer convex optimization
problems.
Few recent works address the parallel or distributed solution of deterministic
mixed-integer programs.  A decentralized (parallel)---but not
distributed---approach, based on dual decomposition, is proposed in
\cite{VUJANIC2016144,falsone_decentralized_2017} to approximately solve
mixed-integer linear programs (MILPs) with guaranteed suboptimality bounds.  \MC{A central unit is used in the two parallel methods above to update and then broadcast back to agents the dual variables. 
} 
In \cite{falsone2018distributed} a
distributed version of the algorithm in \cite{falsone_decentralized_2017} is proposed. 
\MC{In \cite{camisa2018primal} a distributed algorithm based on primal decomposition is proposed for the same MILP set-up.}
  A Lagrange relaxation approach combined with proximal bundle has been used in
  \cite{kim2013scalable} to solve a demand response problem involving
  mixed-integer constraints.  
    In \cite{FRANCESCHELLI2013156}, the problem of heterogeneous multi-vehicle
  routing is formulated as a MILP and a decentralized algorithm, based on a
  gossip protocol, is proposed to find a feasible \emph{suboptimal} solution
  almost surely. A problem of cooperative distributed robust trajectory
  optimization is addressed in \cite{kuwata_cooperative_2011}. The problem is
  formulated as \MC{a} MILP and solved in a receding horizon framework where
  sequentially each vehicle solves its subproblem---optimizing its own control
  input---and perturbs the decision of other subproblems to improve the global
  cost.
Finally\MC{,} a distributed algorithm based on generation of cutting planes and
exchange of active constraints is proposed in \cite{testa2018distributed} to
solve deterministic MILPs.

The main contribution of this paper\MC{---which extends the conference paper  \cite{CHAMANBAZ20174973}---}is the design of a fully distributed
algorithm to solve uncertain mixed-integer programs in a network with directed,
asynchronous, and possibly unreliable, communication. The problem under
investigation is a mixed-integer program in which the constraint set is the
intersection of local uncertain constraints, each one known only by a single
node.  Starting from a deterministic constraint exchange idea introduced in
\cite{notarstefano2011distributed}, the algorithm proposed in this paper
introduces a randomized, sequential approach in which each node \emph{locally}: (i) 
performs a probabilistic verification step (based on a ``local'' sampling of its
uncertain constraint set), and (ii) solves a  deterministic optimization
problem with a limited number of constraints. If suitable termination conditions
are satisfied, we are able to prove that the nodes reach consensus on a common solution\MC{,} 
which is probabilistically feasible and optimal with high confidence. We also
propose a stopping criterion ensuring that the distributed algorithm \MC{can in fact} 
be stopped after a finite number of communication rounds. Although tailored to
mixed-integer programs, the proposed algorithm can solve a more general class of
problems known as \emph{S-optimization} problems \cite{de_loera_beyond_2015}
(which include continuous, integer and mixed-integer optimization problems).
As compared to the literature reviewed above, the proposed algorithm has three main
advantages. First, no assumptions are needed on the probabilistic nature of the
local constraint sets. Second, each node can sample locally its own uncertain
set. Thus, no central unit is needed to extract samples and no common constraint
set needs to be known by the nodes. Third and final, nodes do not need to
perform the whole sampling at the beginning and subsequently solve the
(deterministic) optimization problem. Online extracted samples are used only for
verification, which is computationally inexpensive. The optimization is
always performed on a number of constraints that remains constant for each node
and depends only on the dimension of the decision variable and on the number of
neighboring nodes. Furthermore, the distributed algorithm can be immediately
implemented using existing off-the-shelf optimization tools, and no ad-hoc
implementation of specific update rules is needed.

The paper is organized as follows. In
Section~\ref{sec:preliminaries_and_problem_setup}, we provide some preliminaries
and we formulate the distributed robust mixed-integer convex program.
Section~\ref{sec: distributed algorithm} presents our distributed, randomized
algorithm for finding a solution---with probabilistic robustness---to robust
distributed mixed-integer problems. The probabilistic convergence properties of
the distributed algorithm are investigated in Section~\ref{sec: analysis of the
  algorithm}. Finally, extensive numerical simulations are performed in
Section~\ref{sec: numerical simulations} to show the effectiveness of the
proposed methodology.

\section{Preliminaries \& Problem Set-Up}\label{sec:preliminaries_and_problem_setup}
In this section\MC{,} we introduce some preliminary notions of combinatorial
optimization that will serve us to set-up the Robust Mixed-Integer Convex
Program addressed in the paper and the methods used for the proposed algorithm.

\subsection{Preliminaries}
Given a constraint set
$\cSet{F} = \cSet{F}^1\cap\dots\cap\cSet{F}^N\subset\mathbb{R}^d$  and a vector
$c\in\mathbb{R}^d$, we will denote with $(\cSet{F},c)$ the following
optimization problem over $S\subset\mathbb{R}^d$
\begin{align}\nonumber
  \min\,\,& c^T\decision\\ \nonumber
  \text{subject to }& \decision \in \cSet{F} ,\\ \nonumber 
  & \decision\in S.
\end{align}
Moreover, let us denote  $J(\cSet{F})$ the optimal cost of the above problem.
For a large class of problems, \GN{often known as abstract programs or LP-type
problems}, a solution to $(\cSet{F},c)$ can be identified by considering a subset
of the constraints $\cSet{F}^1,\dots,\cSet{F}^N$. This concept is characterized
by the notion of \emph{basis}, which is, informally, a ``minimal'' set of
constraints that defines a solution.
The concept of basis is supported by Helly-type theorems, initially introduced
by Eduard Helly in \cite{helly1923mengen}, see also
\cite{amenta_helly-type_1994}. Before formally defining a basis, we define the
\emph{Helly number} and its variation \emph{$S$-Helly
  number}~\cite{amenta_hellys_2015,de_loera_beyond_2015}.

\begin{definition}[Helly number]
  Given a nonempty family $\cSet{K}$ of sets, Helly number
  $h = h(\cSet{K})\in\mathbb{N}$ of $\cSet{K}$ is defined as the smallest number
  satisfying the following
    \begin{align*}
      \forall i_1,\dots,i_h\in&\{1,\dots,m\} : K_{i_1}\cap\dots\cap K_{i_h} \neq \emptyset \Rightarrow \\
      &K_{1}\cap\dots\cap K_{m} \neq \emptyset
    \end{align*}
  for all $m\in\mathbb{N}$ and $K_1,\dots,K_m\in\cSet{K}$. If no such $h$ exists then $h(\cSet{K}) = \infty$.
\end{definition}
\GN{For the classical Helly's theorem, $\cSet{K}$ is a finite family of convex subsets
of $\mathbb{R}^d$.} One of the important extensions of the Helly number is the
\emph{$S$-Helly number} defined for $S\subset \mathbb{R}^d$.\oprocend

\begin{definition}[$S$-Helly number]
 \GN{ Given a set $S\subset\mathbb{R}^d$, let $S\cap\cSet{K}$ be the family of sets
  $S\cap K$ such that $K\subset\mathbb{R}^d$ is convex. The $S$-Helly number is defined as
  $h(S) = h(S\cap\cSet{K})$.} \oprocend
 \end{definition}

It is worth pointing out the main difference between the definitions given
above. The Helly number states that if the intersection of every $h$ sets is
nonempty, then there is a point in common between all the sets. Differently, the
$S$-Helly number states that if the intersection of any $h$ convex sets contains
a point in $S$, then all the sets have a point in common in $S$.

\begin{definition}[Basis]
  \GN{Given a constraint set $\cSet{F} = \cSet{F}^1\cap\cdots\cap\cSet{F}^N$, a
  basis $\cSet{B}$ of $(\cSet{F}, c)$ is a minimal collection of constraints
  from $\cSet{F}^1,\dots,\cSet{F}^N$, such that the optimal cost of the problem
  $(\cSet{F}, c)$ is identical to the one of $(\cSet{B}, c)$, i.e.,
  $J(\cSet{F}) = J(\cSet{B})$.\oprocend }
\end{definition}

\GN{We point out that in \MC{this} paper we are slightly abusing \MC{the} notation since we are
  denoting by $\cSet{B}$ both the collection of constraints (when referring to
  the basis) and their intersection (when denoting the constraint set of the
  optimization problem).}

The size of a largest basis of problem $(\cSet{F}, c)$ is called its
\emph{combinatorial dimension}. The following result connects the combinatorial
dimension with the $S$-Helly number. This result is presented in \cite[Theorem
2]{calafiore2012mixed}, \cite[Theorem 1.3]{de_loera_beyond_2015} and
\cite[Theorem 3.11]{amenta_hellys_2015}.

\begin{theorem}
Let $\cSet{F} = \cSet{F}^1\cap\!\cdots\!\cap\cSet{F}^N$, then the combinatorial
dimension of problem $(\cSet{F}, c)$ is $h(S)-1$.\oprocend
\end{theorem}

  It is worth noticing that if in problem $(\cSet{F}, c)$, the sets
  $\cSet{F}^1,\dots,\cSet{F}^N$ are convex, well known classes of problems arise
  depending on $S$.  If $S=\bR^d$, problem $(\cSet{F}, c)$ is a
  continuous convex optimization problem. If $S=\bZ^d$, then
  $(\cSet{F}, c)$ becomes an integer optimization problem. Choosing
  $S=\bZ^{d_Z}\times\bR^{d_R}$ with $d=d_Z+d_R$ leads to a mixed-integer convex
  problem. The S-Helly number for $S=\bZ^{d_Z}\times\bR^{d_R}$ is given below, see
  \cite{averkov2012transversal}.
  \begin{theorem}[Mixed-Integer Helly Theorem]\label{thm:S-helly dimension}
    The Helly number $h(\bZ^{d_Z}\times\bR^{d_R})$ equals $(d_R+1)2^{d_Z}$.\oprocend
  \end{theorem}
  This implies that the combinatorial dimension of $(\cSet{F}, c)$ with
  $S=\bZ^{d_Z}\times\bR^{d_R}$ is $(d_R+1)2^{d_Z} - 1$.

\subsection{Problem Set-Up}
We consider a network of $n$ processors with limited computation and/or
communication capabilities aiming at cooperatively solving the \MC{following} Robust
Mixed-Integer Convex Program (RMICP)  
\begin{align}\nonumber
  \min\,\,& c^T\decision \nonumber\\
  \text{subject to }& \decision\in\bigcap_{i=1}^n \FF^i(q),\,\,\forall q\in\bQ, \label{eq:RMICP}\\
  & \decision\in S,\nonumber
  \end{align}
  where $\decision\in S$ is the vector of decision variables, $q\in\bQ$ is
  the vector of uncertain parameters acting on the system, while
  $\FF^i(q)=\{\decision\in S:f^i(\decision,q)\leq0\}$ is the constraint set known only by agent $i$, with
  $f^i(\decision,q): \mathbb{R}^d\times\bQ\rightarrow\bR$
  its related constraint function, which is assumed to be convex for any fixed value of $q\in\mathbb{Q}$. 
  The objective function is considered to be linear. This assumption is without
  loss of generality. In fact, a nonlinear convex objective function can be
  transformed into the epigraph form by introducing an extra decision variable.

  We make the following assumption regarding the solution of any deterministic
  subproblem of~\eqref{eq:RMICP}, i.e., a problem in which only a finite number
  of constraints $\cSet{F}^i(q)$ (for a given $q$) are considered.
  \begin{assumption}[Non-degeneracy]~\label{assum: nondegeneracy} The minimum
    point of any subproblem of~\eqref{eq:RMICP} with at least $h(S)-1$
    constraints is unique and there exist only $h(S)-1$ constraints intersecting at the
    minimum point.\oprocend
  \end{assumption}
  Assumption~\ref{assum: nondegeneracy} is not restrictive. In fact, to ensure
  uniqueness of the optimal point we could use a strictly convex objective
  function, a lexicographic ordering, or any universal tie-breaking rule, see
  \cite[Observation 8.1]{amenta_helly-type_1994} for further details.
  
  Processor $i$ has only knowledge of the constraint set $\cSet{F}^i(q)$ and of
  the objective direction $c$, which is common among all  nodes. The goal is
  to design a fully distributed algorithm---consisting of purely local computation and
  communication with neighbors---used by each agent in order to cooperatively
  solve problem \eqref{eq:RMICP}. We want to stress that there is no central
  node having access to all constraints.

  We let the nodes communicate according to a time-dependent, directed
  communication graph $\mathcal{G}(t)=\{\mathcal{V},\mathcal{E}(t)\}$ where
  $t\in\natural$ is a universal time which does not need to be known by nodes,
  $\VV=\until{n}$ is the set of agent identifiers and $(i,j)\in\EE(t)$ indicates
  that $i$ sends information to $j$ at time $t$. The time-varying set of incoming
  (\MC{resp.} outgoing) neighbors of node $i$ at time $t$, $\mathcal{N}_\text{in}(i,t)$
  ($\mathcal{N}_\text{out}(i,t)$), is defined as the set of nodes from (\MC{resp.} to)
  which agent $i$ receives (\MC{resp.} transmits) information at time $t$.  A directed
  static graph is said to be \emph{strongly connected} if there exists a
  directed path (of consecutive edges) between any pair of nodes in the
  graph. For time-varying graphs we use the notion of \emph{uniform joint strong
    connectivity} formally defined next.

\begin{assumption}[Uniform Joint Strong Connectivity]~\label{assum:graph}
There exists an integer $L\ge1$ such that the graph
$\bigg(\mathcal{V}, \bigcup_{\tau=t}^{t+L-1} \mathcal{E}(\tau)\bigg)$ is strongly
  connected for all $t\ge0$. \oprocend
\end{assumption}
\vspace{1ex}

There is no assumption on how uncertainty $q$ enters problem \eqref{eq:RMICP},
thus making its solution particularly challenging. In fact, if the uncertainty
set $\mathbb{Q}$ is an uncountable set, problem \eqref{eq:RMICP} is a
semi-infinite optimization problem involving an infinite number of constraints.  In
general, there are two main paradigms to solve an uncertain optimization problem
of the form \eqref{eq:RMICP}. The first approach is a deterministic worst-case
paradigm in which the constraints are enforced to hold for \emph{all} possible
uncertain parameters in the set $\mathbb{Q}$. This approach is computationally
intractable for cases where uncertainty does not appear in a ``simple'' form,
e.g. affine, multi-affine, or convex. Moreover, since some uncertainty scenarios
are very unlikely to happen, the deterministic paradigm may be overly
conservative.  The second approach---\MC{the one} pursued in this paper---is a
probabilistic approach where uncertain parameters are considered to be random
variables and the constraints are enforced to hold for the entire set of
uncertainty except a subset having an arbitrary small probability measure.

\section{Randomized Constraints Consensus}\label{sec: distributed algorithm}
In this section, we present a distributed, randomized algorithm for solving
problem \eqref{eq:RMICP} in a probabilistic sense.
Note that, since the \MC{uncertain} constraint sets in \eqref{eq:RMICP} are
uncountable, it is in general very difficult to verify if a candidate solution
is feasible for the entire set of uncertainty or not.
We instead use a randomized approach based on Monte Carlo simulation\MC{s} to check
probabilistic feasibility. 

\subsection{Algorithm description}
\label{subsec:algorithm_description}

The distributed algorithm we propose has a probabilistic nature and consists of
two main steps: verification and optimization. The main idea is the following. A
node has a candidate basis and candidate solution point. First, it verifies if
the candidate solution point belongs to its local uncertain set with high
probability. Then, it collects bases from neighbors and solves a convex
mixed-integer problem with its basis and its neighbors' bases as constraint
set. If the verification step is not successful, the first violating constraint
is also considered in the optimization.

Formally, we assume that $q$ is a random variable and a probability measure
$\Pr$ over the Borel $\sigma-$algebra of $\mathbb{Q}$ is given.
We denote by $k_i$ a local counter keeping track of the number of times the
verification step is performed by agent $i$. 

In the verification step, each agent $i$ generates a \GN{multi-sample $\ms_{k_i}^{i}$} with cardinality
$M_{k_i}^{i}$ from the set of uncertainty
\[
\ms_{k_i}^{i} \doteq \{q_{k_i,i}^{(1)},\ldots,q_{k_i,i}^{(M_{k_i}^{i})}\} \in\mathbb{Q}^{M_{k_i}^{i}},
\] 
according to the measure $\Pr$, where
$\mathbb{Q}^{M_{k_i}^{i}}\doteq
\mathbb{Q}\times\mathbb{Q}\times\cdots\times\mathbb{Q}$ ($M_{k_i}^{i}$ times).
Node $i$ checks feasibility of the candidate
solution $\decision^i(t)$ only at the extracted samples (by simply checking the
sign of $f^i(\decision^i(t), q_{k_i,i}^{(\ell)})$ for each extracted sample
$q_{k_i,i}^{(\ell)}$, $\ell\in \{1,\ldots,M_{k_i}^{i}\}$) . If a violation happens, the first
violating sample is used as a \emph{violation certificate}\footnote{In fact,
  more than one violated sample---if they exist---can be returned by the
  verification step. See Remark~\ref{rem:violation Certificate} for more
  details.  }.

In the optimization step, agent $i$ transmits its current basis to all outgoing
neighbors and receives bases from incoming ones. Then, it solves a convex
mixed-integer problem whose constraint set is composed of: \emph{i)} a
constraint generated at the violation certificate (if it exists), \emph{ii)}
its current basis, and \emph{iii)} the collection of bases from all incoming
neighbors.
We define a primitive $[\decision^*,\cSet{B}]=\solvemip(\cSet{F},c)$ which
solves the \emph{deterministic} mixed-integer convex problem defined by the pair
$(\cSet{F},c)$ and returns back the optimal point $\decision^*$ and the
corresponding basis $\cSet{B}$.
Node $i$ repeats these two steps until a termination condition is satisfied,
namely if the candidate basis has not changed for $2nL+1$ times, with $L$
defined in Assumption~\ref{assum:graph}.
The distributed algorithm is formally presented in Algorithm~\ref{alg:
  distributed algorithm}.

\begin{algorithm}[h]
\caption{Randomized Constraints Consensus}
\label{alg: distributed algorithm}
\textbf{Input:}{ \GN{$\cSet{F}^i(q)$, $c$, $\varepsilon_i$, $\delta_i$} }

\textbf{Output:}{ 
$\decision_\text{sol}^i$
}

\textbf{Initialization:\\}
Set $k_i=1$, $[\decision^i(1),\cSet{B}^i(1)] = \solvemip(\cSet{F}^i(q^i_0),c)$
\GN{for some $q^i_0\MC{\in \bQ}$}

\textbf{Evolution:}\\
\textbf{Verification:}
\begin{enumerate}[label=\subscript{V}{{\arabic*}}:]
\item
If $\decision^i(t)=\decision^i(t-1)$, set 
$q^{\texttt{viol}}_{t,i}=\emptyset$ and goto {\bf Optimization}
\item 
Extract 
\begin{equation}\label{eq:sample bound Mk}
M_{k_i}^{i}=\Big\lceil \frac{2.3+1.1\ln k_i+\ln \frac{1}{\delta_i}}{\ln \frac{1}{1-\varepsilon_i}}\Big\rceil
\end{equation}
i.i.d samples
$\ms_{k_i}^i =
\{q_{k_i,i}^{(1)},\ldots,q_{k_i,i}^{(M_{k_i}^{i})}\}$
\item
If $\decision^i(t)\in \FF^i(q_{k_i,i}^{(\ell)})$ for all
$\ell=1,\ldots,M_{k_i}^{i}$, set 
$q^{\texttt{viol}}_{t,i}=\emptyset$; else, set
$q^{\texttt{viol}}_{t,i}$ as the first sample for which 
$\decision^i(t)\notin \FF^i(q^{\texttt{viol}}_{t,i})$
\item
Set $k_i=k_i+1$
\end{enumerate}

\textbf{Optimization:} 
\begin{enumerate}[label=\subscript{O}{{\arabic*}}:]
\item 
Transmit $\cSet{B}^i(t)$ to $j\in\mathcal{N}_\text{out}(i,t)$, acquire bases
from incoming neighbors and set $\cSet{Y}^i(t)\doteq\cap_{j\in\mathcal{N}_\text{in}(i,t)}\cSet{B}^j(t)$
\item
$[\decision^i(t+1),\cSet{B}^i(t+1)]=\solvemip(\cSet{F}^i(q^{\texttt{viol}}_{t,i})\cap \cSet{B}^i(t)\cap \cSet{Y}^i(t),c)$
\item
If  $\decision^i(t+1)$ unchanged for $2nL+1$ times, 
 return  
 $\decision_\text{sol}^i=\decision^i(t+1)$
\end{enumerate}
\end{algorithm}

At this point, it is worth highlighting some key interesting features of the
proposed algorithm.
First, the verification step, even if it may run possibly a large number of
times, consists of simple, inexpensive inequality checks which are not
computationally demanding.
Moreover, we remark that if at some $t$ the candidate solution has not changed, that is
$\decision^i(t)=\decision^i(t-1)$, then $\decision^i(t)$ has successfully
satisfied a verification step and, thus, the algorithm does not perform it
again.
Second, the number of constraints involved in the optimization problem being
solved locally at each node is fixed. Hence, the complexity of the problem does
not change with time.
Third, the amount of data a processor needs to transmit does not depend on the
number of agents, but only on the dimension of the space. Indeed, processor $i$
transmits only $h(S)-1$ constraints at each iteration and for mixed-integer
convex programs this number only depends on the dimension of the space, see
Theorem~\ref{thm:S-helly dimension}.
Fourth and final, the proposed distributed randomized algorithm is completely
asynchronous and works under unreliable communication. Indeed, $t$ is a
universal time that does not need to be known by the processors and the graph
can be time-varying.
Thus, if nodes run the computation at different speeds, this can be modeled by
having no incoming and outgoing edges in that time interval. Similarly, if
transmission fails at a given iteration, this is equivalent to assum\MC{ing} that the
associated edge is not present in the graph at that iteration.

\begin{remark}[Local optimization and initial constraint]\label{rem:original constraints set}
  In the deterministic constraints consensus algorithm presented in
  \cite{notarstefano2011distributed}, at each iteration of the algorithm, each
  node needs to include in the local optimization problem its original
  constraint set. Here, we can drop this requirement because of the uncertain
  nature of the problem handled in the verification step. \oprocend
\end{remark}

\begin{remark}[Multiple Violation Certificates]\label{rem:violation Certificate}
  In the verification step of Algorithm~\ref{alg: distributed algorithm}, we set
  $q^{\textnormal{\texttt{viol}}}_{t,i}$ as the first sample for which
  $\decision^i(t)\notin \FF^i(q^{\textnormal{\texttt{viol}}}_{t,i})$. However, if more than
  one violated sample is found, then, in step $O_2$ of Algorithm~\ref{alg:
    distributed algorithm}, we may include $r>1$ violated constraints in the
  \textnormal{\solvemip}~primitive, thus solving a larger optimization problem. By doing
  this, one expects a faster convergence and hence less communication burden. In
  fact, as we will show in the numerical computations, the number of violated
  samples $r$ constitutes a trade-off between the number of times each node
  communicates with neighbors and the complexity of the local optimization
  problem it needs to solve. \oprocend
  \end{remark}

\subsection{Algorithm Convergence and Solution Probabilistic Guarantees}
\label{sec: analysis of the algorithm}
Here\MC{,} we analyze the convergence properties of the distributed
algorithm and investigate the probabilistic properties of the solution computed
by the algorithm.

  We start by defining the set of all \MC{possible} successive independent random extractions
  of all the sequences $\{\mathbf{q}_{k_i}^i\}_\MC{{k_i=1,\ldots,\infty}}$, $i=1,\dots,n$ as
  follows:
  \begin{align*}
    \mathbb{S} := \big\{\mathbf{q} :\;  \mathbf{q} = & [\mathbf{q}_1,\dots,\mathbf{q}_n],\; \\ & \mathbf{q}_i = \{\mathbf{q}_{k_i}^i\}_\MC{{k_i=1,\ldots,\infty}},\;\; i=1,\dots,n\big\}.
  \end{align*}
  Moreover, we define the set
  \begin{align*}
    \mathbb{S}_{\text{sol}} := \{\mathbf{q}\in\mathbb{S} :\;\; \text{Algorithm
    \ref{alg: distributed algorithm} terminates} \},
  \end{align*}
  which is the domain of the random variables $\mathbf{x}_{\text{sol}}^i$. 
  The probability measure on $\mathbf{x}_{\text{sol}}^i$ is therefore defined on this space, and it is denoted $\mathbb{P}_{\infty}$.
%
Now\MC{,} we are ready to present the first main result of the paper.

\begin{theorem}\label{thm: convergence }
Let Assumptions~\ref{assum: nondegeneracy} and~\ref{assum:graph} hold. Given the
probabilistic levels $\varepsilon_i>0$ and $\delta_i>0$, $i = 1,\ldots,n$, let
$\varepsilon= \sum_{i=1}^n\varepsilon_i$ and $\delta= \sum_{i=1}^n\delta_i$.
Then, the following statements hold:
\begin{enumerate}
\item \label{item:cost_convergence} Along the evolution of Algorithm~\ref{alg: distributed algorithm}, the
  cost $J(\cSet{B}^i(t))$ at each node $i\in\{1,\ldots,n\}$ is monotonically
  non-decreasing, i.e., $J(\cSet{B}^i(t+1))\geq J(\cSet{B}^i(t))$, and converges
  to a common value asymptotically. That is,
  $\lim_{t\rightarrow\infty}J(\cSet{B}^i(t))=\bar{J}$ for all
  $i\in\{1,\ldots,n\}$.
\item \label{item:stopping} If the candidate solution of node $i$,
  $\decision^i(t)$, has not changed for $2Ln+1$ communication rounds, all nodes
  have a common candidate solution, i.e.,
  $\decision_\textnormal{sol}^i = \decision_\textnormal{sol}$ for all $i = 1,\dots,n$.

\item \label{item:guarantee_xsol} The following inequality holds for $\decision_\textnormal{sol}$:
    \begin{align*}
      \mathbb{P}_\infty\!\bigg\{\mathbf{q}\in\mathbb{S}_{\textnormal{sol}} : & \mathbb{P}\bigg\{q\in\mathbb{Q} : \mathbf{x}_{\textup{\text{sol}}}\notin\bigcap_{i=1}^n\mathcal{F}^i(q)\bigg\}\! \\
      & \leq \epsilon \bigg\}\! \geq\! 1 - \delta.
    \end{align*}
\item \label{item:guarantee_basis} Let $\cSet{B}_\textnormal{sol}$ be the basis corresponding to $\decision_\textnormal{sol}$. The following inequality holds for $\cSet{B}_\textnormal{sol}$
\begin{align*}
\hspace{-0pt}\Pr_\infty \!\bigg\{\!\ms\in\mathbb{S}_{\textnormal{sol}}:&\Pr\bigg\{\!q\in\mathbb{Q}
  \!:\!J\big(\cSet{B}_\textnormal{sol}\cap \cSet{F}(q)\big)>J(\cSet{B}_\textnormal{sol})\bigg\}\\
  & \leq\varepsilon \bigg\}\geq 1\!-\!\delta
\end{align*}
\hspace{-6pt} where $\cSet{F}(q)\doteq \bigcap_{i=1}^n \cSet{F}^i(q)$. \oprocend
\end{enumerate}

\end{theorem}

The proof is given in Appendix~\ref{app: proof of theorem}.

We briefly discuss the results of this theorem.
Statement~\ref{item:cost_convergence} shows that agents are increasing their
local cost and asymptotically achieve the same cost. This result gives insights
on the following statement.  Statement~\ref{item:stopping} provides a condition
under which processors can stop the distributed algorithm with the guarantee of
having a common solution, whose probabilistic properties are shown in the next
statements.
In the next subsection, we provide a variation of the algorithm that ensures to
stop the algorithm in \MC{finite time}. We point out, though, that numerical
computations show that the algorithm in the original version always satisfies
the condition of \MC{S}tatement~\ref{item:stopping}.
It is also worth mentioning that for a fixed communication graph the bound
$2nL+1$ in \MC{S}tatement~\ref{item:stopping} to stop the algorithm can be replaced
by $2D+1$ where $D$ is the graph diameter.
Statement~\ref{item:guarantee_xsol} says that the probability of violating the
(common) solution $\decision_\text{sol}$ with a new sample (constraint) is
smaller than $\varepsilon$ and this statement holds with probability at least
$1-\delta$. Similarly, based on \MC{S}tatement~\ref{item:guarantee_basis}, the
probability that the cost $J(\cSet{B}_\text{sol})$ associated to
$\decision_\text{sol}$ increases, i.e., the probability of $\decision_\text{sol}$
not being optimal, if a new sample (constraint) is extracted, is smaller
than $\varepsilon$ and this statement holds with probability at least
$1-\delta$.

 \begin{remark}[S-Optimization]
   The class of problems that Algorithm~\ref{alg: distributed algorithm} can
   solve is not limited to mixed-integer problems. The algorithm is immediately
   applicable to any problem having a finite $S$-Helly number. This class of
   problems have recently been introduced---under the name of
   S-optimization problems---in \cite{de_loera_beyond_2015}.
    It is worth mentioning that \MC{some} classes of non-convex problems have finite
    $S$-Helly number---see \cite{amenta_hellys_2015} and references
    therein---and hence can be solved using Algorithm~\ref{alg: distributed
      algorithm}.   \oprocend
  \end{remark}

\subsection{Modified Algorithm with Finite-Time Convergence
  Guarantee} \label{sec: finite time convergence}
In this subsection, we derive a stopping criterion for Algorithm~\ref{alg:
  distributed algorithm} so that it is guaranteed to be halted in \MC{\emph{finite
time}}.
Indeed, there is no guarantee for the condition in \MC{S}tatement~\ref{item:stopping}
of Theorem~\ref{thm: convergence } to be satisfied (and thus for
Algorithm~\ref{alg: distributed algorithm} to be halted in finite-time), even
though we found this to be the case for all simulations we run (see  Section
\ref{sec: numerical simulations} for more details).

The stopping criterion is derived by borrowing tools from the scenario approach
presented in \cite{calafiore_uncertain_2004,calafiore_scenario_2006}.  In the
scenario approach, the idea is to extract a finite number of samples from the
constraint set of the semi-infinite optimization problem~\eqref{eq:RMICP} and
solve the obtained (deterministic) optimization problem. The developed theory
provides bounds on the number of samples to guarantee a desired probabilistic
robustness.
The main idea of the stopping criterion is the following. Once each
$M_{k_i}^{i}$ (the cardinality of the multi-sample used in the verification step
of Algorithm~\ref{alg: distributed algorithm}) exceeds a (common) threshold
associated to a proper scenario bound, agents stop generating new samples in the
verification step and Algorithm~\ref{alg: distributed algorithm} proceeds with
the multi-sample extracted at the last iteration. Thereby, Algorithm~\ref{alg:
  distributed algorithm} solves (in a distributed way) a scenario problem and in
finite time finds a solution feasible for the extracted samples across all
the nodes. Each node has a flag which is communicated locally among all
neighbors. The flag is initialized to zero and is set to one once $M_{k_i}^{i}$
is greater than or equal to the scenario bound. We remark that using, for
example, the algorithm presented in \cite{farina2019distributed} nodes can
identify---in a fully distributed manner---if all flags across the network are
set to one and thus stop generating new samples.

In~\cite{margellos2016distributed}\MC{,} scenario bounds are proposed for a
distributed robust convex optimization framework in which agents use a different
set of uncertainty scenarios in their local optimization programs.
In the next \MC{L}emma we provide analogous bounds for our problem setup.
To this aim, let us consider the following optimization problem
\begin{align}\nonumber
\min\,\,& c^T\decision\\ \nonumber
\text{subject to }& \decision\in \bigcap_{i=1}^n\bigcap_{\ell=1}^{M_i^\text{scen}} \FF^i(q^{(\ell)}), \\ \label{eq:sampled RMICP}
& \decision\in S,
\end{align}
where $M_i^\text{scen}$ is the smallest integer satisfying
\begin{equation}\label{eq:scenario bound distributed}
\delta_i\geq \sum_{\ell=0}^{h(S)-2} { M_i^\text{scen} \choose
  \ell}\varepsilon_i^\ell(1-\varepsilon_i)^{M_i^\text{scen}-\ell}.
\end{equation}

\begin{lemma}\label{lemma: distributed scenario}
  Given the probabilistic levels $\varepsilon_i,\delta_i>0,\, i= 1,\ldots,n$, let
  $\varepsilon=\sum_{i=1}^{n}\varepsilon_i$, $\delta=\sum_{i=1}^n \delta_i$ and
  $M^\textnormal{scen}=\sum_{i=1}^n M_i^\textnormal{scen}$, where $M_i^\textnormal{scen}$ is the smallest integer satisfying
  \eqref{eq:scenario bound distributed}. If agents formulate the sampled
  optimization problem \eqref{eq:sampled RMICP} and all of them
  agree on an optimal solution $\decision^*$, 
\begin{align*}
  \Pr^{M^\textnormal{scen}}\bigg\{\ms\in\mathbb{Q}^{M^\textnormal{scen}}:&
                 \Pr\bigg\{q\in\mathbb{Q}:\decision^*\notin \bigcap_{i=1}^n\FF^i(q)\bigg\}\\
                &\leq\varepsilon \bigg\}\geq1-\delta, 
\end{align*}
\MC{where $\bQ^{M^\textnormal{scen}}=\bQ\times\bQ\times\ldots\times\bQ$ ($M^\textnormal{scen}$ times) and $\Pr^{M^\textnormal{scen}}$ is the product probability measure on $\bQ^{M^\textnormal{scen}}$. }
\oprocend
\end{lemma}
\vspace{1ex}
We omit the proof since it follows \MC{arguments similar to those in } in
\cite[Proposition 1]{margellos2016distributed}. Specifically, by using the
result in~\cite[Theorem 3]{calafiore2012mixed}\MC{,} one can show that the same
arguments in \cite[Proposition 1]{margellos2016distributed} can be followed by
properly changing the cardinality of the support constraint set. In particular,
in the proof of \cite[Proposition 1]{margellos2016distributed} the (Helly)
number $d+1$ of the (continuous) convex program, with $d$ being the dimension of
the decision variable, can be replaced by the $S$-Helly number $(d_R+1)2^{d_Z}$
(as from Theorem~\ref{thm:S-helly dimension}) of problem~\eqref{eq:RMICP}.

The following Proposition summarizes the convergence properties and the solution
guarantees of the modified algorithm (Algorithm~\ref{alg: distributed algorithm}
with the stopping criterion). We report only the feasibility result
corresponding to \MC{S}tatement~\ref{item:guarantee_xsol} of Theorem~\ref{thm:
  convergence }, but also the equivalent of \MC{S}tatement~\ref{item:guarantee_basis}
can be proven. 
\begin{proposition}\label{lemma: finite time convergence}
  Given the probabilistic levels $\varepsilon_i>0$ and $\delta_i>0$,
  $i = 1,\ldots,n$, let $\varepsilon= \sum_{i=1}^n\varepsilon_i$ and
  $\delta= \sum_{i=1}^n\delta_i$. If nodes stop performing steps $V_2$ and $V_4$
  of Algorithm~\ref{alg: distributed algorithm} when $M_{k_i}^{i}\geq M_i^\textnormal{scen}$ for
  all $i=1,\ldots,n$, then the following statements hold:
\begin{enumerate}
\item nodes agree on a solution $\decision_\textnormal{sol}$ in \MC{finite time};
\item 
the solution $\decision_\textnormal{sol}$ satisfies
\begin{align}\nonumber
\!\!\!\!\!\!\! \Pr^M\bigg\{\ms\in\mathbb{Q}^M:
&\Pr\bigg\{q\in\mathbb{Q}: \decision_\textnormal{sol}\notin \bigcap_{i=1}^n\FF^i(q)\bigg\}\\ \label{eq:finite time convergence result}
&\leq\varepsilon \bigg\}\geq1-\delta,
\end{align}
where $M = \sum_{i} M_{k_i}^{i}$.
\end{enumerate}
\end{proposition} 
\begin{proof}
	  \GN{ First, note that if the \MC{A}lgorithm does not stop due to
    \MC{C}ondition~\ref{item:stopping} of Theorem~\ref{thm: convergence }, then all
    nodes stop generating new samples in finite time. When all nodes have
    stopped generating samples they start solving the following deterministic
    optimization problem
  \begin{align}\nonumber
  \min\,\,& c^T\decision\\ \nonumber
  \text{subject to }& \decision\in \bigcap_{i=1}^n\bigcap_{\ell=1}^{M_{k_i}^i} \FF^i(q^{(\ell)}), \\ \label{eq:scenario proof}
  & \decision\in S,
  \end{align}
    where $M_{k_i}^{i}\geq M_i^\textnormal{scen}$ is the cardinality of the verification
  multi-sample at the stopping point.
  To prove that nodes agree in finite time on an optimal solution of
  \eqref{eq:scenario proof}\MC{,} we resort to the following two arguments. First, as
  from \MC{S}tatement \ref{item:cost_convergence} of Theorem~\ref{thm: convergence },
  the cost is monotonically non-decreasing. Second, the number of constraints in
  the network is fixed, thus the number of possible bases is finite. The proof
  of the first statement, thus, follows \MC{from arguments analogous as those in} 
  \cite{notarstefano2011distributed}. Indeed, when nodes stop generating
  samples, the Randomized Constraints Consensus algorithm becomes deterministic
  and turns out to be a variant of the Constraints Consensus algorithm proposed
  in \cite{notarstefano2011distributed}.
  Due to Assumption~\ref{assum: nondegeneracy}, nodes agree on a
  $\decision_\textnormal{sol}$, which is the unique optimal solution of problem
  \eqref{eq:scenario proof}.
  Thus, the second statement follows by noting that, from Lemma~\ref{lemma:
    distributed scenario}, $\decision_\textnormal{sol}$ satisfies \eqref{eq:finite time
    convergence result} since the number of random constraints at each node
  $M_{k_i}^i$ is grater than or equal to the scenario bound $M_i^\textnormal{scen}$.  }
\end{proof}

\section{Numerical Simulation\MC{s}}\label{sec: numerical simulations}
We test the effectiveness of the distributed algorithm presented in
Section~\ref{sec: distributed algorithm} through extensive numerical
simulations.  To this end, we consider two different problems: \emph{(1)}
randomly generated mixed-integer linear programs (MILP) with uncertain
parameters, and \emph{(2)} distributed convex position estimation in wireless
sensor networks.
\MC{These} two numerical simulations are discussed in the next
subsections.

\GN{Numerical computations are run on a Linux\MC{-}based high performance computing
cluster\footnote{ http://idc.sutd.edu.sg/titan/} with $256$ CPUs. Each node uses
only one CPU of the cluster and
executes Algorithm~\ref{alg: distributed algorithm} in an independent Matlab
environment. The communication is modeled by sharing files between different
Matlab environments over a fixed digraph.
}

\subsection{Uncertain Mixed-Integer Linear Programming}
We randomly generate robust \MC{MILPs} with the
following structure
\begin{align*}
\min&~~c^T\decision\\ 
\text{subject to}&~~ 
(A_i^0+A_i^q)^T\decision\leq b_i,\quad i=1,\dots,n\\
& ~~\decision\in \mathbb{Z}^2\times\mathbb{R}^3, 
\end{align*}
where $A_i^0$ is a fixed (nominal) matrix and $A_i^q$ is an interval matrix---a
matrix whose entries are bounded in given intervals---defining the uncertainty
in the optimization problem. We follow the methodology presented in
\cite{dunham_experimental_1977} in order to generate the pair $A_i^0,~b_i$ and
the objective direction $c$. 
In particular, elements of $A_i^0$ are drawn from a standard Gaussian
distribution\MC{, i.e. zero mean and unit standard deviation}. The entries of $A_i^q$ are bounded in $[-\rho, \rho]$ for
some given $\rho$.
%
The $\ell$-th element of $b_i$ is defined by
$b_{i,\ell} = \gamma \, (\sum_{m=1}^{d}\big(A^0_{\ell m}\big)^2)^{1/2}$. This
procedure generates feasible linear programs. However, introducing discrete
variables in a feasible linear program may lead to an infeasible MILP. To ensure
feasibility of the MILP, we increase the volume of the feasible region using the
parameter $\gamma>1$. In the set of simulations reported here, we set
$\gamma=20$.
The objective direction $c$---which is the same for all nodes---is drawn from
the standard Gaussian distribution. The communication graph $\mathcal{G}$ is a
(connected) random $k$-nearest neighbor graph (with $k$ being the fixed \MC{degree or} number
of neighbors) \cite{frieze2015introduction}.
\GN{Over this fixed graph, agents implement the distributed algorithm according
  to the asynchronous and unreliable communication protocol (based on Matlab
  environments) described above.  
  In particular, only one node at a time is able to read/write from/to each node
  file containing its basis.  Hence, if more than one node is willing to
  read/write from/to a file, only one would be allowed to do so. This gives
  raise to an asynchronous and unreliable communication protocol. }

 \begin{table*}[t]
 \caption{\MC{A}verage---over all nodes---number of times a basis is transmitted to the neighbors,  average number of times verification is performed ($k_i$ at the convergence) and empirical violation of the computed solution ($\decision_\text{sol}$) over $10,000$ random samples for different number of nodes and neighbors in each node. The simulation is performed $100$ times for each row and average results are reported.}
 \begin{center}
 \scalebox{0.9}{
 \begin{tabular}{c|c|c|c||c|c|c}
 \toprule
\# Nodes  & \# Neighbors & Graph&  \# Constraints &  \# Transmissions &  \GN{\# Verifications} & Empirical \tabularnewline
 $n$  &  in each node \MC{(degree)} & diameter & in each node &       (averaged)    & (averaged)    & violation \tabularnewline
 \midrule
 \midrule
 $10$  &  $3$  & $4$ &  $100$  & $29.86$  & $15.6$  & $8.18\times 10^{-4}$   \tabularnewline
 \midrule
 $50$  &  $4$  &  $4$ & $100$  & $37.58$   & $10.73$ & $2\times 10^{-4}$       \tabularnewline
 \midrule
$100$  &  $6$  &  $4$ & $100$  & $41.73$   & $9.79$ & $3\times 10^{-4}$       \tabularnewline
\midrule
$200$  &  $8$  &  $4$ & $100$  & $43.94$   & $8.92$ & $2.3\times 10^{-4}$       \tabularnewline

 \bottomrule
 \end{tabular}
 }
 \end{center}

 \label{tab: simulation results}
 \end{table*}

 We use the \texttt{intlinprog} function of Mosek \cite{andersen2000mosek} to
 solve optimization problems appearing at each iteration of the distributed
 algorithm.
 \GN{As for the basis computation\MC{,} it is worth mentioning that in (continuous)
   convex optimization problems, bases can be found by \MC{identifying} a minimal
   set of active constraints.}
 \MC{For mixed-integer problems, however, finding the bases can be computationally expensive. A computationally manageable procedure for finding a bases set---which might not  be of minimum cardinality---is to check all constraints one by one to see
 whether or not they belong to a basis. To this end, we drop one constraint at a time  and solve the resulting optimization problem. If the optimal objective of this optimization problem is smaller than the objective value of the original problem, we conclude that the dropped constraint is part of the basis. }
However, we note that, this \MC{brute-force} approach is performed on a limited
number of constraints---the constraint formed at the violation certificate,
constraints in the node basis, and constraints in the neighboring bases.


\MC{For} the computations reported in Table~\ref{tab: simulation results}, we varied
the number of nodes and neighbors per node, \GN{while using only those graphs
  \MC{having a} diameter equal to $4$.} The number of constraints per node is set to
$100$. We also consider all elements of $A_q$ to be bounded in $[-0.2,0.2]$. The
underlying probability distribution of uncertainty appearing in $A_q$ is
selected to be uniform, see, e.g., \cite{bai_worst-case_1998}.  The
probabilistic accuracy and confidence levels of each agent are
$\varepsilon_i = 0.1/n$ and $\delta_i = 10^{-9}/n$ respectively, with $n$ being
the number of nodes (first column of Table~\ref{tab: simulation results}).
 We report the average---over all nodes---number of times each node updates its
 basis and transmits it to the outgoing neighbors. It is assumed that each node
 keeps the latest information received from neighbors and hence, if the basis is
 not updated, there is no need to re-transmit it to the neighbors. This also
 accounts for the asynchronicity of the distributed algorithm. 
 We also report the average---over all nodes---number of times node $i$ performs
 the verification step, i.e., \GN{the average\MC{---over all the nodes---}of the $k_i$s at convergence.}
 This allows us to show that with a relatively small number of ``design''
 samples used in step $O_2$ of Algorithm \ref{alg: distributed algorithm}, nodes
 compute a solution with \MC{a} high degree of robustness.
 In order to examine the robustness of the obtained solution, we run an {\em a
   posteriori} analysis based on Monte Carlo simulations. \GN{To this end, we
 check the feasibility of the
 obtained solution for $10,000$ random samples extracted from the uncertain
 sets of the nodes.} The empirical violation (last column of Table~\ref{tab: simulation
   results}) is measured by dividing the number of samples that violate the
 solution by $10,000$.
 \GN{ Since Algorithm~\ref{alg: distributed algorithm} has a stochastic nature,
   for each row of Table~\ref{tab: simulation results} we generate randomly
   $100$ problems, solve them using Algorithm \ref{alg: distributed algorithm},
   run an {\em a posteriori} analysis and then report the average values.}

\begin{figure}
\begin{center}
\includegraphics[width=0.98\columnwidth]{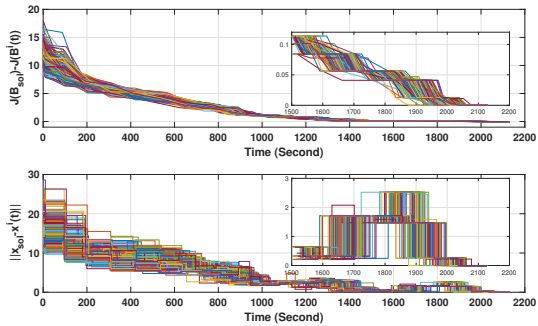}    
\caption{Top: Distance to $J(\cSet{B}_\text{sol})$ which is the objective value
  algorithm converges to. Bottom: Distance to $\decision_\text{sol}$ which is
  the solution algorithm converges to. The two plots are generated for a problem
  instance corresponding to the the last row of Table~\ref{tab: simulation
    results}.}
\label{fig:objective and distance}
\end{center}
\end{figure}

In Figure~\ref{fig:objective and distance}, we report the distance of objective
value $J(\cSet{B}^i(t))$ and candidate solutions
$\decision^i(t), \forall i\in\{1,\ldots,n\}$ from $J(\cSet{B}_\text{sol})$ and
$\decision_\text{sol}$ respectively along the distributed algorithm execution
for a problem instance corresponding to the last row of Table~\ref{tab: simulation results}. 

In the \MC{following} computations\MC{,} we test the effect of using more violation certificates
in the optimization step---see Remark~\ref{rem:violation
  Certificate}. Specifically, we use in the optimization step up to $10$
violating certificates obtained from the verification step.
  As  expect\MC{ed}, this decreases the number of communications required for
  convergence at the expense of the local optimization complexity.  This result
  is reported in Table \ref{tab: simulation results Multiple violations}.
  \begin{table*}[t]
   \caption{\MC{A}verage---over all nodes---number of times a basis is
     transmitted to the neighbors, average number of times verification is
     performed ($k_i$ at the convergence) and empirical violation of the computed
     solution ($\decision_\text{sol}$) over $10,000$ random samples for
     different number of nodes and neighbors in each node. We allow $10$
     violation samples to be returned by the verification step. The simulation
     is performed $100$ times for each row and average results are reported. }
 \begin{center}
 \scalebox{0.9}{
 \begin{tabular}{c|c|c|c||c|c|c}
 \toprule
\# Nodes  & \# Neighbors & Graph&  \# Constraints &  \# Transmissions &  \GN{\#
                                                                       Verifications}
   & Empirical \tabularnewline
 $n$  &  in each node \MC{(degree)}& diameter & in each node &       (averaged)    & (averaged)    & violation \tabularnewline
 \midrule
 \midrule
 $10$  &  $3$  & $4$ &  $100$  & $16.1$  & $8.1$  & $10\times 10^{-3}$   \tabularnewline
 \midrule
 $50$  &  $4$  &  $4$ & $100$  & $24.37$   & $7.05$ & $9\times 10^{-3}$       \tabularnewline
 \midrule
$100$  &  $6$  &  $4$ & $100$  & $26.25$   & $6.3$ & $4\times 10^{-3}$       \tabularnewline

 \bottomrule
 \end{tabular}
 }
 \end{center}

 \label{tab: simulation results Multiple violations}
 \end{table*}

In all $700$ simulations reported in Tables \ref{tab: simulation results} and
\ref{tab: simulation results Multiple violations}, Algorithm~\ref{alg:
  distributed algorithm} converges to a solution with desired probabilistic
robustness in finite time. We further remark that in none of the simulations we
\MC{ran}, the stopping criterion derived in subsection \ref{sec: finite time
  convergence} was met. In fact\MC{,} a closer investigation of the stopping condition
presented in Lemma~\ref{lemma: finite time convergence}---namely
$M_{k_i}^{i}\geq M_i^\text{scen}\text{ for all }i\in\{1,\ldots,n\}$---reveals that this
condition happens only for extremely large values of the verification counter
$k_i$. In \cite[Theorem 4]{alamo_randomized_2015}, an analytical \MC{sub-optimal} solution for
the sample size $M_i^\text{scen}$ in the inequality \eqref{eq:scenario bound distributed} is
provided
\begin{equation}\label{eq:alamo bound}
  M^\text{scen}_i\geq \frac{1.582}{\varepsilon}\bigg(\ln \frac{1}{\delta_i}+h(s)-2\bigg).
\end{equation}
Solving the inequality $M_{k_i}^{i}\geq M^\text{scen}_i$ for $k_i$ and noting that $\ln \big(\frac{1}{1-\varepsilon_i}\big)\simeq\varepsilon_i$ for small $\varepsilon_i$, we obtain 
\[
k_i \geq \exp\left(\frac{0.58\ln\frac{1}{\delta_i}+1.58(h(S)-2)-2.3}{1.1}\right).
\]
The Helly number associated to the MILP \MC{envisaged} here is $16$. Considering,
for example, $\varepsilon_i=0.01$ and $\delta_i = 10^{-10}$, the verification
counter ensuring that $M_{k_i}^{i}\geq M^\text{scen}_i$ becomes
$k_i\geq 1.2\times 10 ^{13}$. By looking at the sixth column of Tables \ref{tab:
  simulation results} \MC{and \ref{tab: simulation results Multiple violations}}\MC{,} which indicates the value of verification counter at
convergence, one can see that the distributed algorithm converges to a solution
with desired probabilistic properties in a much smaller number of verification
steps.

\subsection{Distributed Localization in Wireless Sensor Networks}
\label{sec: application example}
The second numerical example is \MC{a} problem of distributed position estimation
in wireless sensor networks. A centralized version of this problem---with no
uncertainty---is formulated in~\cite{sensor_network}.

Consider a two-dimensional\footnote{Extension to a three-dimensional space is
  straightforward.} space containing a number of heterogeneous wireless sensors
which are randomly placed over a given environment.
The sensors are of two classes. \MC{The first class are wireless sensors with known positions (SwKP)} which are capable of
positioning themselves up to a given accuracy, i.e., with some uncertainty in
their position. They are equipped with processors with limited
computation/communication capabilities and play the role of computational nodes
in the distributed optimization framework. These sensors can communicate with
each other based on a metric distance. That is, two sensors which are close
enough can establish a bidirectional communication link.
\MC{The second class is a  wireless sensor with unknown position (SwUP) which  has no positioning capabilities. This sensor is
  only equipped with a short-range transmitter having an isotropic  communication range.}

\begin{figure}
\centering
\includegraphics[width=0.48\textwidth]{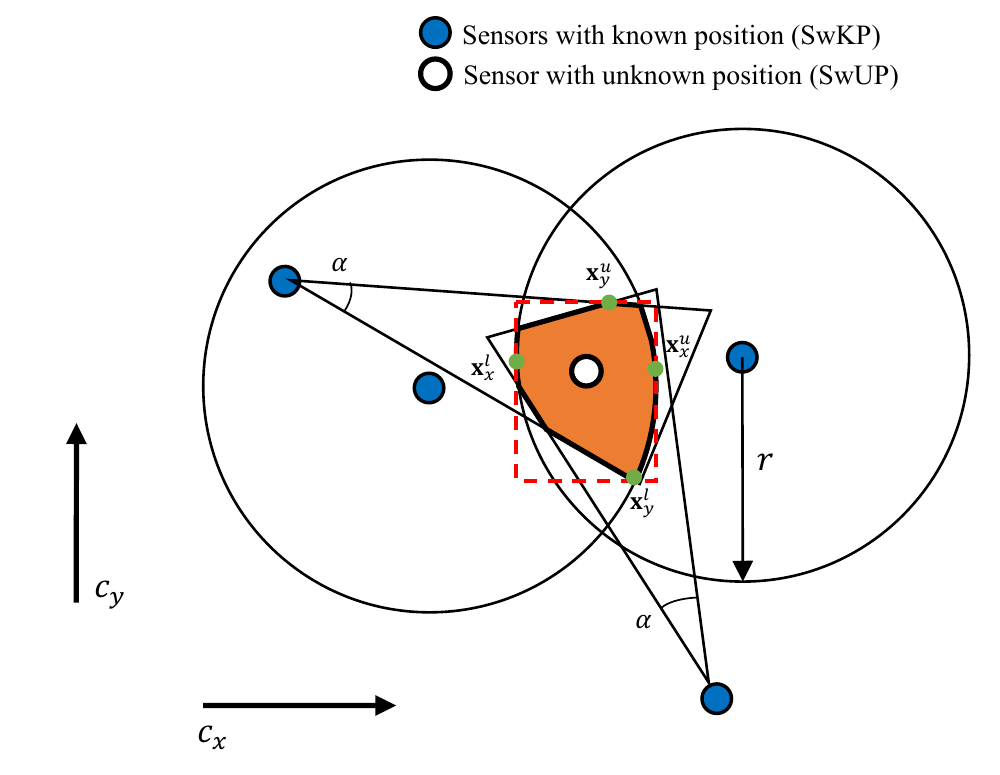}
\caption{Sensors with known position \MC{(SwKP: blue dots)} are computational nodes in the distributed
  optimization setup. Their objective is to estimate the position
    of the sensor with unknown position \MC{(SwUP: white dot)}. 
  The \MC{orange} (shaded) region represents the feasible set for the unknown
  position. 
}
\label{fig: sensors intersection}
\end{figure}

We further consider a heterogeneous setup in which some of the \MC{SwKPs}  are equipped with \MC{a} laser transceiver providing an
estimate of the relative angle with respect to the \MC{SwUP}, which is
within the range of the laser transceiver. This ``angular
constraint'' can be represented by the intersection of three half-spaces
$\FF^i_A\doteq\{\decision\in\mathbb{R}^2: a_k\decision-b_k\leq0,k=1,2,3\}$,
where $\decision\in\mathbb{R}^{2}$ is the position of \MC{SwUP} and
parameters $a_k\in\mathbb{R}^{1\times2}\text{ and }b_k\in\mathbb{R}$ define the
three half-spaces (two to bound the angle and one corresponding to the
transceiver range) as presented in Fig. \ref{fig: sensors intersection}.
\MC{SwKPs}, which are not equipped with laser transceiver, can
place a ``radial constraint'' for the \MC{SwUP} if \MC{it} is within the
communication range.  If the signal from the \MC{SwUP} can be sensed by a
\MC{SwKP}, a proximity constraint exists between them, that is, the \MC{SwUP} is placed in a circle with radius $r$ centered at the \MC{position of SwKP}
 with $r$ being the communication range. This constraint can be
formulated as $\|\decision-p_i\|_2\leq r$, where $p_j\in\mathbb{R}^{2}$ is the
position of $i$th \MC{SwKP}. Using the
Schur-complement~\cite{boyd2004convex}, the feasible set of radial constraint
can be represented as
\[
\FF^i_R\doteq\Bigg\{\decision\in\mathbb{R}^2:\left[\begin{array}{c c}
rI_2 & (\decision-p_i)\\
(\decision-p_i)^T & r
\end{array}\right]\succeq 0\Bigg\},
\]
where $I_2$ is the $2\times2$ identity matrix \MC{and $\succeq$ denotes positive semidefiniteness of the matrix }. 

The objective of the network\MC{ed system} is to ``estimate''---in a distributed fashion---the
position of the \MC{SwUP}. More precisely, \MC{SwKPs} try to find the
smallest box
$\{\decision\in\mathbb{R}^2:
[\decision_x^l,\decision_y^l]\leq\decision\leq[\decision_x^u,\decision_y^u]\}$ \MC{that is}
guaranteed to contain the \MC{SwUP}, i.e., the red dotted box in
Fig. \ref{fig: sensors intersection}.
As shown in Fig. \ref{fig: sensors intersection}, the position of the \MC{SwUP} is often sensed and hence constrained by several \MC{SwKP}. This is
used to shrink the feasibility set of \MC{the position of the SwUP}. We remark that in a
distributed setup, constraints regarding the position of the \MC{SwUP} are
distributed \MC{through} different computational nodes and there is no agent having full
knowledge of all constraints. The bounding box can be computed by solving four
optimization problems with linear objectives. For instance $\decision_x^l$ can
be obtained by solving the following optimization problem
\begin{align*}
\decision_x^l&=\text{argmin } ~\begin{bmatrix}
  1 \\
  0
\end{bmatrix}^T\decision\\
\text{subject to } ~\decision&\in\bigcap_{i=1}^n\bigg(\FF^i_A\cap\FF^i_R\bigg),
\end{align*}
where $n$ is the number of \MC{SwKPs} capable of constraining the position of
\MC{the} \MC{SwUP}.
\GN{We point out that for \MC{SwKPs} that are not equipped with a laser
  transceiver  $\FF^i_A=\mathbb{R}^2$ \MC{holds}.}

\MC{SwKPs} are capable of localizing themselves. Nevertheless, their
localization is not accurate and there is an uncertainty associated with the
position vector reported by each of these nodes. This uncertainty can be modeled as a
norm-bounded vector
\begin{equation}\label{eq:uncertainty radius}
\|p_i-\bar{p}_i\|_2\leq\rho,
\end{equation}
where $\bar{p}_i$ is the nominal position reported by sensor $i$, $p_i$ is its
actual position and $\rho$ is the radius of the uncertain set. In other words,
the uncertainty is defined by an $\ell_2$\MC{-}ball centered at the nominal point
$\bar{p}_i$ with radius $\rho$.

For simulation purposes, we consider a $10\times10$ square environment
containing $n$ \MC{SwKPs} (computational nodes) whose purpose is to estimate
the position of \MC{the} \MC{SwUP}. We remark that \MC{the} extension to the case in which
there is more than one \MC{SwUPs} is straightforward.  

Two \MC{SwKPs}
can establish a two-way communication link if their distance is less than $7$
units of distance. The communication range of \MC{SwUP} is also considered
to be $7$ units. Among the $n$ \MC{SwKPs}, half of them are able to place
angular constraints---using laser transceiver---on their relative angle to the
\MC{SwUP} within the laser range. The angle representing the accuracy of
the laser transceiver---the angle $\alpha$ in Fig.~\ref{fig: sensors
  intersection}---is $20$ degrees. Finally, the uncertainty radius $\rho$ in
\eqref{eq:uncertainty radius} is $0.1$ and the distribution of the uncertainty
is selected to be uniform. The probabilistic accuracy and confidence parameters
are selected to be $\varepsilon_i = 0.1/n$ and $\delta_i = 10^{-9}/n$ for all
$i=1,\ldots,n$, leading to $\varepsilon=0.1$ and $\delta = 10^{-9}$.

\begin{table}[t]
 \caption{The average---over all nodes---number of times a basis is transmitted to the neighbors,  average number of times verification is performed ($k_i$ at the convergence) and empirical violation of the computed solution ($\decision_\text{sol}$) over $10,000$ random samples for different number of known sensors.}
 \begin{center}
 \scalebox{0.9}{
 \begin{tabular}{c|c|c|c}
 \toprule
\# Nodes   &    \# Transmissions &  \GN{\# Verifications} & Empirical \tabularnewline
 $n$   &        (averaged)    & (averaged)    & violation \tabularnewline
 \midrule
 \midrule
 $10$   & $ 7.77$  & $8.77$  & $1\times 10^{-4}$   \tabularnewline
 \midrule
 $50$  &   $12.48$   & $7.36$ &$0$       \tabularnewline
 \midrule
$100$  &  $9.37$   & $6.52$ & $0$       \tabularnewline
\midrule
$200$  &  $7.95$   & $7.42$ & $0$       \tabularnewline
 \bottomrule
 \end{tabular}
 }
 \end{center}

 \label{tab: simulation results sensor network}
 \end{table}

 The distributed robust localization problem is solved for different values of
 the number of \MC{SwKPs} $n$. The result is reported in Table~\ref{tab:
   simulation results sensor network}. Similar to Tables \ref{tab: simulation
   results} and \ref{tab: simulation results Multiple violations}, we report the
 average---over all $n$ nodes---number of times nodes transmit their basis and
 run the verifications step. We also run an {\em a posteriori} analysis \MC{with} 
 Monte Carlo simulation\MC{s} using $10,000$ random samples to check the robustness of
 the computed solution. The empirical violation reported by the Monte Carlo
 simulation\MC{s} for each formulated scenario is reported in the last column of
 Table~\ref{tab: simulation results sensor network}.  We remark that the number
 of uncertain parameters is $2n$, thus rendering the problem nontrivial to
 solve. For instance, considering the last row of Table~\ref{tab: simulation
   results sensor network}, where $n=200$, there are $400$ uncertain parameters
 in the problem.

 To conclude, we would like to point out an interesting feature of
 Algorithm~\ref{alg: distributed algorithm}. In the results presented in
 Tables~\ref{tab: simulation results}, \ref{tab: simulation results Multiple
   violations}, and \ref{tab: simulation results sensor network}, one can
 observe that when the number of computational nodes increases \MC{while} keeping the graph
 diameter constant, the number of transmissions required for convergence does
 not change significantly. This suggests that the number of transmissions
 required for convergence is independent from the number of computational nodes
 but rather depends on the graph diameter.
   %

\section{Conclusions}\label{sec: conclusions}
In this paper, we proposed a randomized distributed algorithm for solving robust
mixed-integer convex programs in which  constraints  are scattered across a network
of processors communicating by means of a directed time-varying graph. The
distributed algorithm has a sequential nature consisting of two main steps:
verification and optimization. Each processor iteratively verifies a candidate
solution through a Monte Carlo \MC{simulations}, and solves a local MICP whose constraint
set includes its current basis, the collection of bases from neighbors and
possibly, a constraint---provided by the Monte Carlo algorithm---violating the
candidate solution. The two steps, i.e. verification and optimization, are
repeated \MC{until} a local stopping criterion is met and all nodes converge to a
common solution. We analyzed the convergence properties of the proposed
algorithm.

\appendix

\section{Proof of Theorem~\ref{thm: convergence }}\label{app: proof of theorem}
\textit{Proof of \MC{S}tatement~\ref{item:cost_convergence}.}
\GN{The proof of the first statement follows \MC{arguments very similar to those in}  \cite{burger2014polyhedral} and is \MC{thus} omitted.
}

\textit{Proof of \MC{S}tatement~\ref{item:stopping}.}
\GN{The proof of this statement heavily relies on results in
  \cite{hendrickx2008graphs}, but we report it here for completeness.}
Since the graph is uniformly jointly strongly connected, for any \MC{pair} of nodes
$u$ and $v$ and for any $t_0>0$, there exists a time-dependent path from $u$ to
$v$ \cite{hendrickx2008graphs}---a sequence of nodes $\ell_1,\ldots,\ell_k$  and a sequence of time
instances $t_1,\ldots,t_{k+1}$ with $t_0\leq t_1<\ldots<t_{k+1}$, such
that the directed edges
$\{(u,\ell_1),(\ell_1,\ell_2),\ldots,(\ell_k,v)\}$ \MC{belong} to the
directed graph at time instances $\{t_1,\ldots,t_{k+1}\}$, respectively. 
We now prove that the path from $u$ to $v$ is of length at most $nL$.
We recall that $n$ is the number of nodes and $L$ is defined in
Assumption~\ref{assum:graph}. 
To this end, following~\cite{hendrickx2008graphs}, define a time-varying set
$S_t$ and initialize it with the node $u$ itself, i.e. $S_{t_0} = {u}$. Next,
given $S_t$ construct $S_{t+1}$ by adding all nodes $\ell_{t+1}$ such that
$\ell_{t+1}\in\NN_\text{out}(\ell_t,t+1)$ with $\ell_{t}$ being all the nodes in
the set $S_t$. Now, notice that if $S_t$ does not contain all nodes, i.e.,
$S_t\neq \VV$ then after $L$ time instants at least one new node must be added
to the set $S_t$, otherwise taking the union graph over the $L$ time instants,
the set $S_t$ and that node would not be connected by any path, thus violating
Assumption~\ref{assum:graph}. Therefore, in at most $nL$ time instants $S_{t}$
will contain all nodes including $v$. Meaning that the length of time-varying
path connecting $u$ and $v$ is at most $nL$.
Consider nodes $i$ and $p$. If
$\ell_{1}\in\mathcal{N}_\text{out}(i,t_0)$, then
$J(\cSet{B}^i(t_0))\leq J(\cSet{B}^{\ell_1}(t_0+1))$ as the constraint set of node
$\ell_{1}$ at time $t_0+1$ is a superset of the constraint set of node $i$
at time $t_0$. Iterating this argument, we obtain
$J(\cSet{B}^i(t_0))\leq J(\cSet{B}^p(t_0+nL))$.
Again since the graph is uniformly jointly strongly connected, there will be a time varying path of length at most $nL$ from node $p$ to node $i$. Therefore,
\begin{equation*}
J(\cSet{B}^i(t_0))\leq J(\cSet{B}^p(t_0+nL))\leq J(\cSet{B}^i(t_0+2nL)).
\end{equation*}
If $J(\cSet{B}^i(t_0))=J(\cSet{B}^i(t_0+2nL))$ and considering the point that
node $p$ can be any node of the graph, then all nodes have the same cost. That
is, $J(\cSet{B}^1(t))=\ldots=J(\cSet{B}^n(t))$. This combined with
Assumption~\ref{assum: nondegeneracy} proves \MC{S}tatement~\ref{item:stopping}.

\textit{Proof of \MC{S}tatement~\ref{item:guarantee_xsol}.}
We start by defining the following violation sets:
  \begin{align*}
    \text{Viol}_i &:= \bigg\{q\in\mathbb{Q} :\; \mathbf{x}_{\text{sol}} \notin \mathcal{F}^i(q) \bigg\}, \qquad i=1,\dots,n,\\
    \text{Viol} &:= \bigg\{q\in\mathbb{Q} :\; \mathbf{x}_{\text{sol}} \notin \bigcap_{i=1}^n \mathcal{F}^i(q) \bigg\}.
  \end{align*}
  We observe that
  \begin{align*}
    \mathbf{x}_{\text{sol}} \notin \bigcap_{i=1}^n \mathcal{F}^i(q) \iff \exists i = 1,\dots,n :\; \mathbf{x}_{\text{sol}} \notin \mathcal{F}^i(q),
  \end{align*}
  hence
  \begin{equation*}
    \text{Viol} = \bigcup_{i=1}^n \text{Viol}_i,
  \end{equation*}
  therefore
  \begin{align*}
    \mathbb{P}\big\{\text{Viol}\big\} = \mathbb{P}\bigg\{\bigcup_{i=1}^n \text{Viol}_i \bigg\} \leq \sum_{i=1}^n \mathbb{P}\big\{\text{Viol}_i\big\}.
  \end{align*}
  At this point, we note that, applying simple properties coming from the probability theory, we have
  \begin{align*}
    \mathbb{P}_\infty\bigg\{\mathbf{q}&\in\mathbb{S}_{\text{sol}} :\; \mathbb{P}\bigg\{q\in\mathbb{Q} : \mathbf{x}_{\textup{\text{sol}}}\notin\bigcap_{i=1}^n\mathcal{F}^i(q)\bigg\} \leq \epsilon \bigg\}\\ 
    &= \mathbb{P}_\infty\bigg\{\mathbf{q}\in\mathbb{S}_{\text{sol}} :\; \mathbb{P}\big\{\text{Viol}\big\} \leq \epsilon \bigg\}\\
    &\geq \mathbb{P}_\infty\bigg\{\mathbf{q}\in\mathbb{S}_{\text{sol}} :\; \sum_{i=1}^n \mathbb{P}\big\{\text{Viol}_i\big\} \leq \epsilon\bigg\}\\
    &= \mathbb{P}_\infty\bigg\{\mathbf{q}\in\mathbb{S}_{\text{sol}} :\; \sum_{i=1}^n \mathbb{P}\big\{\text{Viol}_i\big\} \leq \sum_{i=1}^n \epsilon_i\bigg\}\\
    &\geq \mathbb{P}_\infty\bigg\{\bigcap_{i=1}^n \bigg\{\mathbf{q}\in\mathbb{S}_{\text{sol}} :\; \mathbb{P}\big\{\text{Viol}_i\big\} \leq \epsilon_i\bigg\}\bigg\}\\
    &\geq 1 - \sum_{i=1}^n \mathbb{P}_\infty\bigg\{\mathbf{q}\in\mathbb{S}_{\text{sol}} :\; \mathbb{P}\big\{\text{Viol}_i\big\} > \epsilon_i\bigg\}.
  \end{align*}
  As a consequence, if we prove that
  \begin{align}\label{eq:final}
    \mathbb{P}_\infty\big\{\mathbf{q} \in \mathbb{S}_{\text{sol}} :\; \mathbb{P}\big\{\text{Viol}_i\big\} > \epsilon_i \big\} \leq \delta_i,\qquad i=1,\dots,n,
  \end{align}
  then the thesis will follow. 

  Now, let $i=1,\dots,n$ and $k\MC{_i}\in\mathbb{N}$. We define the events:
  \begin{align*}
    &\text{ExitBad}_i := \big\{\mathbf{q} \in \mathbb{S}_{\text{sol}} :\; \mathbb{P}\big\{\text{Viol}_i\big\} > \epsilon_i \big\},\\
    &\text{Feas}_{k_i}^i := \big\{ \mathbf{q}\in\mathbb{S}_{\text{sol}} :\; \mathbf{x}_{\text{sol}}\in\mathcal{F}^i(q_{k_i,i}^{(\ell)}),\;\;\forall \ell=1,\dots,M_{k_i}^i \big\}.
  \end{align*}
  We immediately note that \eqref{eq:final} is equivalent to:
  \begin{align}\label{eq:final1}
    \mathbb{P}_{\infty}\big\{\text{ExitBad}_i\big\} \leq \delta_i,\qquad i=1,\dots,n.
  \end{align}
  Thus we will prove \eqref{eq:final1} instead of \eqref{eq:final}.

  First, we observe that if $\mathbf{q}\in\text{ExitBad}_i$, then \MC{there} exists $k_i$ such that $\mathbf{q}\in\text{Feas}_{k_i}^i$. Consequently, the following relation holds:
  \begin{align*}
    \text{ExitBad}_i &\subset \bigcup_{k_i=1}^\infty \big( \text{Feas}_{k_i}^i\cap \text{ExitBad}_i \big).
  \end{align*}
  Therefore, we have 
  \begin{align}\label{eq:series}
    \mathbb{P}_{\infty}\big\{\text{ExitBad}_i\big\} \leq \sum_{k_i=1}^\infty \mathbb{P}_{\infty}\big\{\text{Feas}_{k_i}^i \cap \text{ExitBad}_i \big\},
  \end{align}
  and we can bound the terms of the last series as follows:
  \begin{small}
    \begin{align}
      \mathbb{P}_{\infty}\big\{& \text{Feas}_{k_i}^i\cap \text{ExitBad}_i \big\}\nonumber\\
      &= \mathbb{P}_{\infty}\big\{ \text{Feas}_{k_i}^i \big| \text{ExitBad}_i \big\} \mathbb{P}_{\infty}\big\{ \text{ExitBad}_i \big\}\label{eq:iid2}\\
      &\leq \mathbb{P}_{\infty}\big\{ \text{Feas}_{k_i}^i \big| \text{ExitBad}_i \big\}\label{eq:iid3}\\
      &= \mathbb{P}_{\infty}\bigg\{ \mathbf{q}\in\mathbb{S}_{\text{sol}} :\; \mathbf{x}_{\text{sol}}\in\mathcal{F}^i(q_{k_i,i}^{(\ell)}),\;\;\forall \ell=1,\dots,M_{k_i}^i\;\;\nonumber\\
      &\hspace{2.5cm}\bigg|\; \mathbb{P}\big\{ q\in\mathbb{Q} : \mathbf{x}_{\text{sol}} \notin \mathcal{F}^i(q) \big\} > \epsilon_i \bigg\}\label{eq:iid4}\\
      &= \mathbb{P}_{\infty}\bigg\{ \mathbf{q}\in\mathbb{S}_{\text{sol}} :\; \mathbf{x}_{\text{sol}}\in\mathcal{F}^i(q_{k_i,i}^{(\ell)}),\;\;\forall \ell=1,\dots,M_{k_i}^i\;\;\nonumber\\
      &\hspace{2.5cm} \bigg|\; \mathbb{P}\big\{ q\in\mathbb{Q} : \mathbf{x}_{\text{sol}} \in \mathcal{F}^i(q) \big\} \leq 1 - \epsilon_i \bigg\}\label{eq:iid5}\\
      &\leq (1 - \epsilon_i)^{M^i_k}.\label{eq:iid8}\vspace{-2pt}
    \end{align}
  \end{small}
\hspace{-5pt} A few comments are given in order to explain the above \MC{set of equalities and inequalities}: equation \eqref{eq:iid2} follows from the chain rule; inequality \eqref{eq:iid3} is due to the fact that $\mathbb{P}_{\infty} \leq 1$; in equation \eqref{eq:iid4} we have written explicitly the events $\text{Feas}_{k_i}^i$ and $\text{ExitBad}_{i}$; equation \eqref{eq:iid5} exploits the property $\mathbb{P}\big\{\mathbb{Q}-A\big\} = 1 - \mathbb{P}\big\{ A \big\}$; inequality \eqref{eq:iid8} is due to the fact that in order to declare $\mathbf{x}_{\text{sol}}$
  feasible, the algorithm needs to perform $M^i_{k_i}$
  independent successful trials each having \MC{a} probability of success smaller than or equal to $(1-\epsilon_i)$.

  At this point, \MC{combining} \eqref{eq:iid8} \MC{and} \eqref{eq:series}, we obtain
  \begin{align*}
    \mathbb{P}_{\infty}\big\{ \text{ExitBad}_i \big\} \leq \sum_{k=1}^\infty (1 - \epsilon_i)^{M^i_{k_i}},
  \end{align*}
    and the last series can be made arbitrarily small by an appropriate choice of $M^i_{k_i}$. In particular, by choosing
  \begin{align}\label{eq:zeta}
    (1 - \epsilon_i)^{M^i_{k_i}} = \frac{1}{k_i^\alpha}\frac{1}{\xi(\alpha)}\delta_i,
  \end{align}
  where $\xi(\alpha) := \sum_{s=1}^{\infty} \frac{1}{s^\alpha}$ is the Riemann Zeta function evaluated at $\alpha > 1$, we have
  \begin{align*}
    \sum_{k_i=1}^{\infty} (1 - \epsilon_i)^{M^i_{k_i}} = \sum_{k_i=1}^{\infty} \frac{1}{k_i^\alpha}\frac{1}{\xi(\alpha)}\delta_i = \frac{1}{\xi(\alpha)}\delta_i \sum_{k_i=1}^{\infty} \frac{1}{k_i^\alpha} = \delta_i.
  \end{align*}
  Hence, the choice of sample size is obtained by solving \eqref{eq:zeta} for $M^i_{k_i}$:
  \begin{align*}
    M^i_{k_i} = \bigg\lceil\frac{\alpha\log(k_i) + \log(\xi(\alpha)) + \log\big(\frac{1}{\delta_i}\big)}{\log\big(\frac{1}{1-\epsilon_i}\big)}\bigg\rceil.
  \end{align*}
  Note that the optimal value of $\alpha$ minimizing $M^i_{k_i}$ is computed by trial and error to be $\alpha=1.1$; by means of this choice of $\alpha$, we obtain the number $M^i_{k_i}$ specified in the algorithm, and we can thus conclude that \eqref{eq:final1} is true, and so the thesis follows.

\textit{Proof of \MC{S}tatement~\ref{item:guarantee_basis}.}
Define the following two sets
\begin{align*}
V \doteq &\bigg\{q\in\mathbb{Q}:\decision_\text{sol}\notin \bigcap_{i=1}^n\FF^i(q)\bigg\}\\
W \doteq &\bigg\{q\in\mathbb{Q}:J\big(\cSet{B}_\text{sol}\cap \cSet{F}(q)\big)>J(\cSet{B}_\text{sol})\bigg\},
\end{align*}
where $\cSet{F}(q)\doteq \bigcap_{i=1}^n \cSet{F}^i(q)$. We first prove that $V=W$. 

Let $q_v\in V$, that is, $\decision_\text{sol}\notin\bigcap_{i=1}^n\FF^i(q_v)$, then $J(\cSet{B}_\text{sol}\cap \cSet{F}(q_v))\geq J(\cSet{B}_\text{sol})$ with $\cSet{F}(q_v)\doteq\bigcap_{i=1}^n\cSet{F}^i(q_v)$. Also, due to Assumption~\ref{assum: nondegeneracy},  any subproblem of~\eqref{eq:RMICP} has a unique minimum point and hence, $J(\cSet{B}_\text{sol}\cap \cSet{F}(q_v))\neq J(\cSet{B}_\text{sol})$. Hence, $q_v\in W$ and subsequently $V\subseteq W$.

Now let $q_v\in W$ that is $J\big(\cSet{B}_\text{sol}\cap \cSet{F}(q_v)\big)>J(\cSet{B}_\text{sol})$ and suppose by contradiction that $\decision_\text{sol}\in \FF(q_v)$, i.e. $q_v\notin V$. Considering the \MC{fact} that $\cSet{B}_\text{sol}$ is the basis corresponding to $\decision_{\text{sol}}$, we conclude that $\decision_{\text{sol}}\in\mathcal{B}_\text{sol}\cap \mathcal{F}(q_v)$ implying that $J(\cSet{B}_\text{sol}\cap \cSet{F}(q_v))\leq J(\cSet{B}_\text{sol})$.  This contradicts the fact that $J\big(\cSet{B}_\text{sol}\cap \cSet{F}(q_v)\big)>J(\cSet{B}_\text{sol})$. Hence, $q_v\in V$ and subsequently, $W\subseteq V$. Since $V\subseteq W$ and $W\subseteq V$, therefore $V=W$.

This argument combined with the result of \MC{Statement~\ref{item:guarantee_basis}} proves the fifth statement of the theorem. That is, the probability that the solution $\decision_\text{sol}$  is no longer optimal for a new sample  equals the probability that the solution is violated by the new sample.  
%

\bibliographystyle{plain}

\end{document}